\documentclass[runningheads]{llncs}
\usepackage{enumitem}   
\usepackage{amsmath,amssymb,amsfonts}
\usepackage{graphicx}
\usepackage{xcolor}

\usepackage{nopageno}

\newtheorem{thm}{Theorem}

\newtheorem{lem}[thm]{Lemma}
\newtheorem{conj}[thm]{Conjecture}
\newtheorem{prob}[thm]{Problem}
\newtheorem{obs}[thm]{Observation}

\newtheorem{defin}[thm]{Definition}

\date{}
\title{\bf Decomposition of Geometric Graphs\\ into Star-Forests\thanks{J\'anos Pach's Research partially supported by European Research Council (ERC), grant “GeoScape” No.\ 882971
and by the Hungarian Science Foundation (NKFIH), grant K-131529. Work by Morteza Saghafian is partially supported by the European Research Council (ERC), grant No.\ 788183, and by the Wittgenstein Prize, Austrian Science Fund (FWF), grant No.\ Z 342-N31.}}

\author{J\'anos Pach\inst{1,2}\orcidID{0000-0002-2389-2035} \and
Morteza Saghafian\inst{3}\orcidID{0000-0002-4201-5775} \and
Patrick Schnider\inst{4}\orcidID{0000-0002-2172-9285} }

\authorrunning{J. Pach et al.}

\institute{R\'enyi Institute of Mathematics, Budapest, Hungary \\
\email{pach@cims.nyu.edu}\and
ISTA (Institute of Science and Technology
Austria), Klosterneuburg, Austria \\
\email{pach@cims.nyu.edu}\and
ISTA (Institute of Science and Technology
Austria), Klosterneuburg, Austria \\
\email{morteza.saghafian@ist.ac.at}\and
Department of Computer Science, ETH Z\"{u}rich, Switzerland\\
\email{patrick.schnider@inf.ethz.ch}}

\begin{document}
	\maketitle
	
\begin{abstract}
We solve a problem of Dujmovi\'c and Wood (2007) by showing that a complete convex geometric graph on $n$ vertices cannot be decomposed into fewer than $n-1$ star-forests, each consisting of noncrossing edges. This bound is clearly tight. We also discuss similar questions for abstract graphs.  
    \end{abstract}

\section{Introduction}\label{intro}
To determine the smallest number of subgraphs of some special kind that a graph $G$ can be partitioned into is a large and classical theme in graph theory. In particular, the parts may be required to be matchings (as in Vizing's theorem~\cite{Vi64}), complete bipartite graphs (as in the Graham-Pollak theorem~\cite{GrPo71}), paths and cycles (as in Lov\'asz' theorem~\cite{Lo68}), forests (as in the Nash-Williams theorem~\cite{Nas64}), etc. 
\smallskip

Most likely, it was Erd\H os who first realized that one can ask many interesting new extremal questions for graphs drawn in the plane or in some other surface, if we replace the purely combinatorial conditions by geometric ones; see~\cite{Pa13}. For instance, we may require that the edges participating in a matching or a path do not cross each other~\cite{AvH66}, \cite{Ku79}. In the 80s and 90s, the emergence of Graph Drawing as a separate discipline gave fresh impetus to this line of research. 
\smallskip

A \emph{geometric graph} $G$ is a graph whose vertex set is a set of points in the plane, no 3 of which are collinear, and whose edges are (possibly crossing) line segments connecting certain pairs of vertices. If the vertices of $G$ are in \emph{convex position}, that is, they form the vertex set of a convex polygon, then $G$ is called a \emph{convex geometric graph}. In the sequel, whenever we say that a graph or a geometric graph $G$ \emph{can be decomposed} into certain parts, we mean that its \emph{edge set}, $E(G)$, can be partitioned into such parts. Each part can be regarded as a different \emph{color class} in the corresponding coloring. 

A \emph{star} is a graph consisting of a vertex together with some edges incident to it. In particular, a single vertex is counted as a star. A graph whose every connected component is a star is called a \emph{star-forest}. The (edge set of a) complete graph $K_n$ with $n$ vertices can be decomposed into $n-1$ stars. Akiyama and Kano~\cite{AkK85} proved that fewer stars do not suffice. (This also follows from the Graham-Pollak theorem~\cite{GrPo71}, mentioned above.) However, it was also shown in~\cite{AkK85} that one can decompose $K_n$ into much fewer \emph{star-forests}: one needs only $\lceil n/2\rceil + 1$ of them. Can one also decompose a \emph{complete convex geometric graph} on $n$ vertices into fewer than $n-1$ star-forests, if we insist that each star-forest is a \emph{plane graph}, that is, its edges do not cross each other? This question was raised by Dujmovi\'c and Wood~\cite{DuW07} (Section 10).

The aim of this note is to answer this question in the negative.

\begin{thm}\label{Thm1}
Let $n\ge 1$. The complete convex geometric graph with $n$ vertices cannot be decomposed into fewer than $n-1$ plane star-forests.
\end{thm}

On the other hand, there are complete geometric graphs where fewer than $n-1$ plane star-forests suffice: consider $P=A_1\cup A_2\cup A_3\cup A_4$ a point set consisting of four pairwise disjoint sets $A_1,\ldots,A_4$, each of size $k$, such that for every choice $P_1\in A_1,\ldots,P_4\in A_4$ we have that $P_4$ lies inside the convex hull of $P_1, P_2$ and $P_3$. Then, it can be seen that the complete geometric graph on $P$ can be decomposed into $3k=3n/4$ plane star-forests, which come in three families: the first family consists of stars emanating from points in $A_1$ connecting to all points in $A_1$ and $A_2$ together with stars emanating from points in $A_3$ connecting to all points in $A_3$ and $A_4$. Similarly, we draw stars emanating from points in $A_2$ connecting to all points in $A_2$ and $A_3$ and from points in $A_4$ connecting to all points in $A_4$ and $A_1$, and for the last family stars from points in $A_1$ connecting to all points in $A_1$ and $A_3$ and from points in $A_2$ connecting to all points in $A_2$ and $A_4$.

The most important unsolved question in this direction is, how much the bound in Theorem~\ref{Thm1} can be improved if we drop the assumption that the vertices are in convex position. We conjecture that the above example is optimal.

\begin{conj}\label{conj:star_forests}
Let $n\ge 1$. There is no complete geometric graph with $n$ vertices that be decomposed into fewer than $\lceil 3n/4\rceil$ plane star-forests.
\end{conj}

Note that in the example above, all star-forests had exactly two components.
A star-forest consisting of at most $k$ connected components (stars) is said to be a \emph{$k$-star-forest}. 

It is also an interesting open problem to determine the minimum number of plane $k$-star-forests that a complete (convex) geometric graph of $n$ vertices can be decomposed into. We do not even know the answer to the analogous question for abstract graphs.

\begin{prob}\label{Pr1}
Let $k$ and $n$ be fixed positive integers. What is the minimum number of $k$-star-forests that a complete graph $K_n$ of $n$ vertices can be decomposed into?
\end{prob}

As was mentioned earlier, for $k=1$, the minimum is $n-1$. The following result settles the first nontrivial case.

\begin{thm}\label{Thm2}
  The complete graph with $n>3$ vertices can be decomposed into $\lceil 3n/4\rceil$ 2-star-forests. This bound cannot be improved.
\end{thm}

In particular, this shows that any counterexample to Conjecture \ref{conj:star_forests} would require the use of star-forests with more than 2 components.

Many other variants of decomposing complete geometric graphs have been studied in the literature, including decompositions into plane spanning trees. The conjecture that every complete geometric graph on $2m$ vertices can be decomposed into $m$ plane spanning trees has been recently disproved in~\cite{AOOPSSTV}. Several notions of thickness studied in~\cite{DuW07} are concerned with decompositions of graphs into plane substructures. For many other interesting questions on abstract and geometric graph parameters, consult~\cite{ArDH05}.

In Sections~\ref{sec:2} and \ref{sec:3}, we prove Theorems~\ref{Thm1} and \ref{Thm2}, respectively.

\section{Covering with plane star-forests--Proof of Theorem~\ref{Thm1}}\label{sec:2}
Recall that a \emph{plane star-forest} is a star-forest which is a plane graph, i.e., its edges do not cross each other.
In this section, in a slight abuse of notations, we will denote the complete convex geometric graph on $n$ points as $K_n$.
Instead of \emph{decompositions} of $K_n$ into plane star-forests, it will be more convenient to consider \emph{coverings}, that is, to allow an edge to belong to more than one star-forest (to have more than one ``color''). This does not change the problem, because by keeping just one color for each edge, we turn any covering of the edge set of $K_n$ into a decomposition.

\begin{defin}
A collection of plane star-forests, $F_1,F_2, \ldots, F_t$ forms a \emph{covering} of $K_n$ if every edge of $K_n$ belongs to \emph{at least one} $F_i$.
\end{defin}

For the proof, we need to introduce some simple terminology. The graphs consisting of just one vertex or a single edge are also regarded as stars. Every star $S$ has a \emph{center}. If $S$ is a vertex, then it is its own center. If $S$ is a single edge, we arbitrarily fix one of its endpoints and call it the center of $S$. The center of a star $S$ is also said to be the \emph{center of any edge} of $S$. Accordingly, if $F$ is a (plane) star-forest, we always assume that each of its components is a star with a \emph{fixed center}. 

\begin{proof}[Proof of Theorem~\ref{Thm1}]
For $n=1,2$, the statement is trivial. Assume for contradiction and let $n\ge 3$ be the smallest number for which the statement is not true. Let $K_n$ be a complete convex geometric graph, and denote its vertices by $P_1, P_2,\ldots, P_n$, in clockwise order. The indices are taken modulo $n$, so that $P_{n+1}=P_1, P_{n+2}=P_2$, etc.

Suppose that $K_n$ is covered by $t$ plane star-forests, $F_1, F_2,\ldots, F_{t}$, for some $t<n-1$. Our goal is to move some edges from one star-forest to another (i.e., to ``recolor'' them) in order to turn at least one $F_i$ into a single star. We make sure that after each step of this process, we obtain a covering of $K_n$ with plane star-forests. As soon as one of the $F_i$s becomes a single star, we remove its center from $K_n$, and contradict with $n$ being the smallest number for which we have a covering of $K_n$ with fewer than $n-1$ plane star-forests.

For every $a, \, 1\le a\le n,$ and for every $k, \, 1<k<n$, we call the edge $P_aP_{a+k}$ a $k$-edge. Note that every $k$-edge is also a $(n-k)$-edge. 

\begin{defin}\label{supported}
    A $k$-edge $P_aP_{a+k}$ is called \emph{supported} if there exists $F_i$ such that $P_aP_{a+k}$ belongs to $F_i$, and
    \begin{enumerate}[label=(\roman*)]
    \item either all edges $P_aP_{a+1},P_aP_{a+2},\ldots,P_aP_{a+k-1}$ belong to $F_i$,
    
    \item or all edges $P_{a+1}P_{a+k},P_{a+2}P_{a+k},\ldots,P_{a+k-1}P_{a+k}$ belong to $F_i$.
    \end{enumerate}
Otherwise, we call it \emph{unsupported}.
\end{defin}

The goal is to recolor the edges step by step in order to make all the edges supported. For this purpose, the following observation is useful for the recoloring process.

\begin{obs}\label{movingedges}
Suppose that the complete geometric graph $K_n$ is covered by $t$ plane star-forests, $F_1,F_2,\ldots F_t$. Let $S$ be a connected component of $F_i$ (that is, a star) where $1 \leq i \leq t$. Assume that no edge of $S$ crosses an edge of $F_j$ where $1 \leq j \leq t$, $j \neq i$. Remove the edges in $S$ from $F_i$ and add them to $F_j$. Then any edge that was supported before is still supported.
\end{obs}

\begin{lem}\label{CLAIM} Suppose that the complete geometric graph $K_n$ can be covered by $t$ plane star-forests, for some positive integer $t$.

Then, for every $k$, $1<k<n$, there exists a covering of $K_n$ by $t$ plane star-forests $F_1,F_2\ldots, F_{t}$ such that every $k'$-edge with $1< k'\leq k$ is supported.
\end{lem}

\begin{proof} We prove the lemma by induction on $k$. 

Suppose that $k=2$. By symmetry, it is sufficient to consider the $2$-edge $P_1P_3$ (that is, $a=1)$. We can assume without loss of generality that $P_1P_3$ belongs to $F_i$, for some $i$, and its center is $P_1$ (which implies that $P_2P_3$ is not in $F_i$). If $P_1P_2$ belongs to $F_i$, condition (i) in Definition \ref{supported} is satisfied, and we are done. If $P_1P_2$ does not belong to $F_i$, then add it to $F_i$. Obviously, it cannot cross any other edge in $F_i$. The only problem that may occur is that until now $P_2$ was a single vertex star in $F_i$, and now $F_i$ has two stars that have a point in common. In this case, simply erase the single vertex star $P_2$ from $F_i$. Thus, the lemma is true for $k=2$.




\smallskip

Suppose next that $k>2$ and 
the statement has already been verified for $k-1$. We want to prove it for $k$. 

By symmetry, it is enough to consider the $k$-edge $P_1P_{k+1}$ and make it supported without making the already supported edges unsupported. Suppose without loss of generality that $P_1P_{k+1}$ belongs to a star in $F_i$ and the center of this star is $P_1$. The edges in $F_i$ are marked \emph{blue}. 
\smallskip

Let $l<k+1$ be the \emph{largest} index such that $P_1P_l$ does \emph{not} belong to $F_i$. Then the edges $P_1P_{l+1}, \ldots, P_1P_{k+1}$ are all blue. If there is no such index $l$, then we are done, because $F_i$ satisfies condition (i) in Definition \ref{supported}.

By the induction hypothesis the edge $P_1P_l$ is supported, so there exists a star-forest $F_j,\, j\neq i,$ which contains $P_1P_l$ along with all the edges $P_1P_2, P_1P_3, \ldots, P_1P_{l-1}$ or along with all the edges $P_2P_l,$ $P_3P_l,\ldots,$ $P_{l-1}P_l$. The edges of $F_j$ are marked \emph{red}. We distinguish two cases depending on these two possibilities.
\smallskip

\noindent\textbf{Case 1:} \emph{The edges $P_1P_2,P_1P_3,\ldots,P_1P_{l}$ belong to $F_j$.}

We make two changes. See Figure 1. 
\smallskip
\begin{figure}[htb]
  \centering \vspace{0.0in}
  \includegraphics[width=1.00\textwidth]{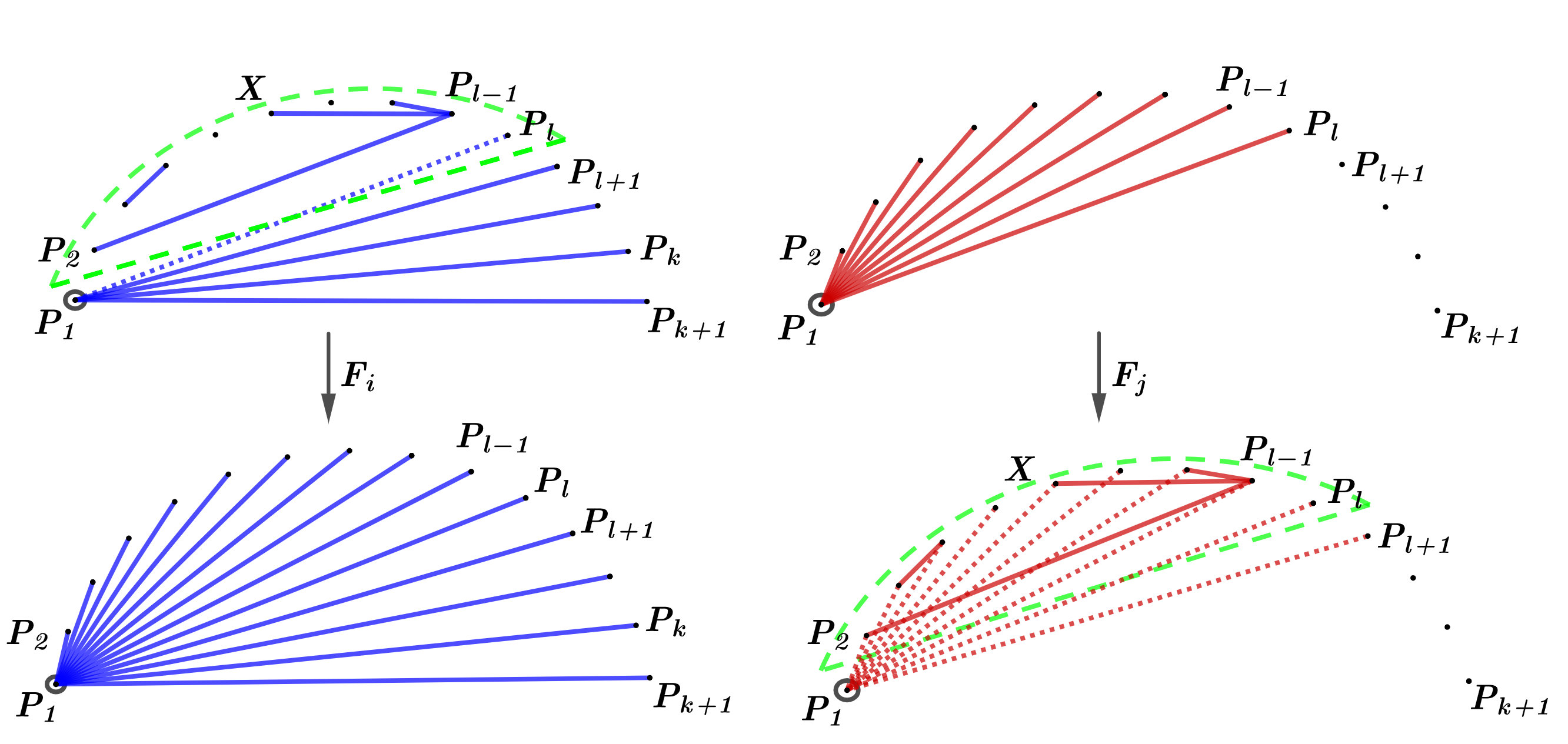}
  \caption{In Case 1, recolor $P_1P_2,\ldots, P_1P_l$ from red to blue, and all blue stars spanned by $\{P_2,\ldots, P_l\}$ to red. A dotted line marks the absence of an edge.}
  \label{fig:Case1}
\end{figure}

\noindent \textsc{Step 1}: Remove the edges $P_1P_2,P_1P_3, \ldots, P_1P_l$ from $F_j$ and add all of them to $F_i$ (unless they were already in $F_i$).

Then $P_1P_{k+1}$ will satisfy condition (i) of definition \ref{supported} in $F_i$ (with $a=1$). However, in the process, we may have created some crossings within $F_i$, and $F_i$ may also cease to be a star-forest. Both of these problems can be avoided by performing
\smallskip

\noindent \textsc{Step 2}: Remove from $F_i$ all (blue) edges connecting two elements of $\{P_2,P_3,\ldots,P_l\}$ and add them to $F_j$. 


Note that by recoloring the blue edges within $\{P_2,P_3,\ldots,P_l\}$ to red, we do not violate the condition that $F_j$ is a plane star-forest. Indeed, unless $l=2$, originally, no element of $\{P_2,P_3,\ldots,P_l\}$ was connected by a red edge to any vertex other than $P_1$. Also by Observation \ref{movingedges}, neither of the two steps results in any previously supported edge becoming unsupported.
\smallskip

\noindent\textbf{Case 2:} \emph{The edges $P_2P_l,P_3P_l,\ldots,P_{l-1}P_l$ belong to $F_j$.}

First, we will modify $F_i$ by including the edge $P_1P_l$. This will require some care, to make sure that the new covering does not violate the conditions. See Figure 2.
\smallskip
\begin{figure}[htb]
  \centering \vspace{0.0in}
  
  \includegraphics[width=1.00\textwidth]{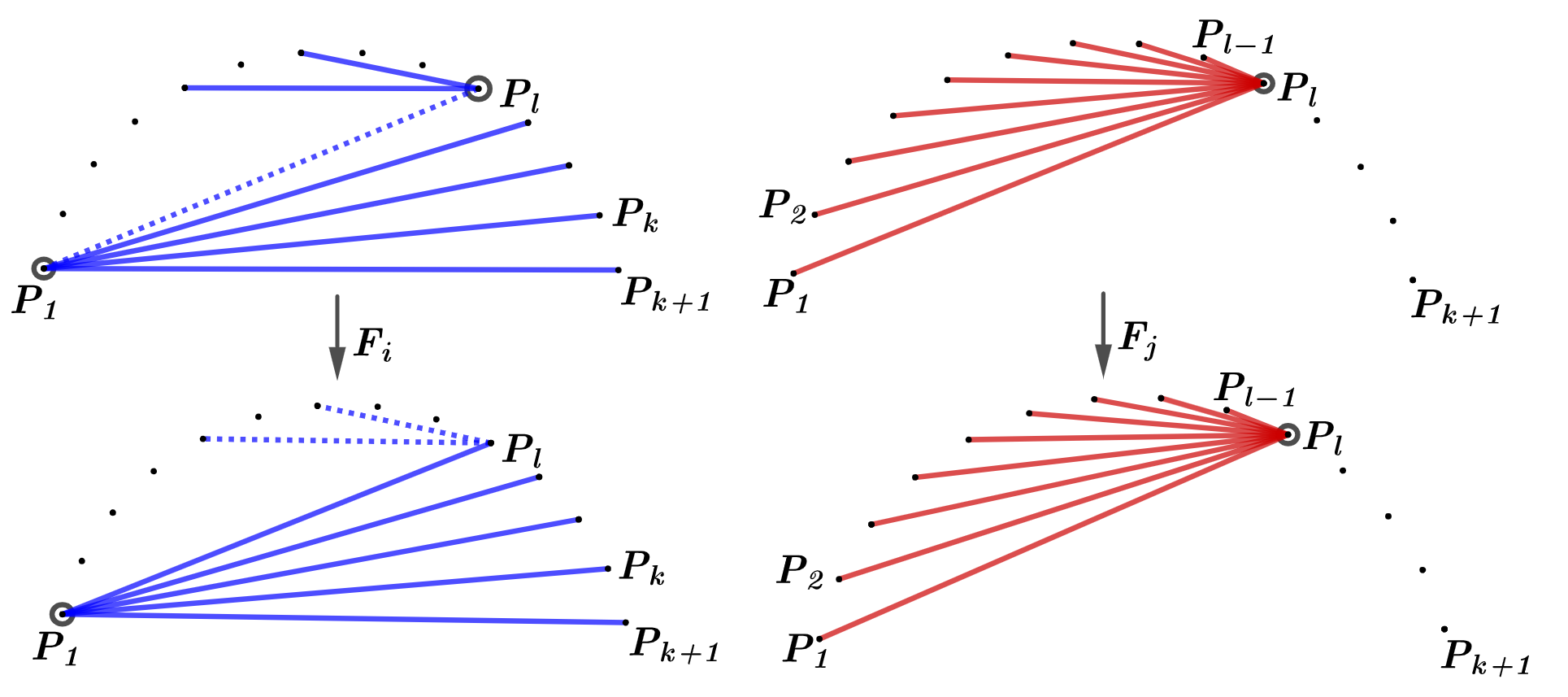}
  \caption{In Case 2, $P_1P_l$ will have two colors: red and blue. Remove the color blue from all previously blue edges incident to $P_l$.}
  \label{fig:Case2}
\end{figure}

\noindent \textsc{Step 1:}  Add the edge $P_1P_l$ to $F_i$, but also keep it in $F_j$. Remove from $F_i$ all other edges incident to $P_l$. 
\smallskip

Notice that after performing this step, we still have a covering of $K_n$ by plane star-forests. It is a \emph{covering}, because all edges deleted from $F_i$ also belonged, and continue to belong, to $F_j$. Obviously, $F_i$ remains a \emph{star-forest}: its component containing $P_1$ remains a star, because we removed from $F_i$ any other edge incident to $P_l$. Finally, $F_i$ remains a \emph{plane graph}, because its newly added edge, $P_1P_l$ cannot cross any other blue edge. Indeed, such an edge should be incident to $P_{l+1}$, contradicting our assumption that $P_1P_{k+1}$ originally belonged to a star in $F_i$, whose center is $P_1$.
Also note that edges incident to $P_l$ in $F_i$ form a connected component which is already in $F_j$. So removing them is equivalent to recoloring them as red, which, by Observation \ref{movingedges}, does not make any already supported edge unsupported.

\smallskip

Now we go back to the beginning of the proof, and again find the largest index $l'$ such that $P_1P_{l'}$ does not belong to $F_i$. Obviously, we have $l'<l$. As before, we distinguish two cases. In Case 1, we conclude that $P_1P_{k+1}$ satisfies condition (i) of Definition \ref{supported} in $F_i$ (with $a=1$), and we are done with the induction step. In Case 2, we can include the edge $P_1P_{l'}$ in $F_i$. Continuing like this, in fewer than $k$ steps, we arrive at a situation where either $P_1P_{k+1}$ satisfies condition (i) of definition \ref{supported} in $F_i$, or one by one, we manage to include all of the edges $P_1P_{k+1}, P_1P_k, \ldots, P_1P_3, P_1P_2$ in $F_i$, which again means that $P_1P_{k+1}$ satisfies condition (i) of definition \ref{supported} in $F_i$.
This completes the proof of Lemma~\ref{CLAIM}.
\end{proof}

Applying the lemma with $k=n-1$ and $a=1$, we can construct a covering of $K_n$ by fewer than $n-1$ plane star-forests such that one of them, again denoted by $F_i$, has the property that either $P_1P_2, P_1P_3,\ldots, P_1P_n$ belong to $F_i$, or $P_1P_n, P_2P_n, \ldots, P_{n-1}P_n$ belong to $F_i$. That is, $F_i$ is a single star of degree $n-1$, centered at $P_1$ or $P_n$. Deleting $P_1$ or $P_n$, resp., from $K_n$, we obtain a covering of $K_{n-1}$ with fewer than $n-2$ plane star-forests, which contradicts our assumption that Theorem~\ref{Thm1} is true for decompositions and, hence, for coverings of the complete convex geometric graph $K_{n-1}$. This completes the proof of Theorem~\ref{Thm1}.
\end{proof}

\section{2-Star-Forests--Proof of Theorem~\ref{Thm2}}\label{sec:3}

\begin{proof}
Let $V$ be an $n$-element set, and let $V=V_1\cup V_2\cup V_3\cup V_4$ be a partition of $V$ into $4$ subsets as equal as possible. Suppose without loss of generality that
\[\lfloor n/4\rfloor\le |V_1|\le|V_2|\le|V_3|\le|V_4|\le\lceil n/4\rceil.\]
Let $f : V_2\rightarrow V_1$ be a surjection (onto mapping). For every $u\in V_2$, consider the two-star-forest $F_u$ consisting of all edges connecting $u$ to a every vertex in $V_2\cup V_3$, and connecting $f(u)$ to every vertex in $V_1\cup V_4$. These two-star-forests completely cover all edges within $V_2$ and $V_1$, and all edges in $V_1\times V_4$ and in $V_2\times V_3$. In a similar manner, we can construct $|V_4|$ two-star-forests that cover all edges within $V_4$ and $V_3$, and all edges in $V_4\times V_2$ and $V_3\times V_1$. Finally, with $|V_3|$ two-star-forests (with one center in $V_3$ and one in $V_1$), we can cover all edges in $V_3\times V_4$ and $V_1\times V_2$. Thus, we covered $K_n$ with $|V_2|+|V_3|+|V_4|=\lceil 3n/4\rceil$ two-star-forests, as required.
\smallskip

Next, we show that $K_n$ cannot be covered by fewer than $\lceil 3n/4\rceil$ two-star-forests, for any $n\ge 4$. The case $n = 4$ is easy. The proof is by contradiction. Let $n$ be the smallest value greater than $4$ for which there exists a covering of $K_n$ by $t \leq \lceil 3n/4\rceil -1$  two-star-forests. Denote the two-star-forests participating in such a covering by $F_1,\ldots, F_t$. If any $F_i$ has only one center, then deleting it from $K_n$, together with all edges incident to it, we reduce the number of vertices by $1$ and the number of two-star-forests by $1$. This would contradict the minimal choice of $n$. Thus, we can and will assume that every $F_i,\ 1\le i\le t,$ has two centers.




Now consider a graph $G$ with the same set of vertices as $K_n$, and for every 2-star-forest $F_i$, draw an edge in $G$ between the two centers of stars in $F_i$. The resulting graph $G$ has at most $\lceil 3n/4\rceil -1$ edges and, therefore, at least $n - \lceil 3n/4\rceil + 1$ connected components. Note that $3(n - \lceil 3n/4\rceil + 1) > \lceil 3n/4\rceil -1$, so there exists a connected component $C$ in $G$ with fewer than $3$ edges.

If $C$ is a single vertex $u$, then by construction it cannot be the center of any two-star-forest. Thus, we would need at least $n-1$ two-star-forests just to cover the edges incident to $u$ in $K_n$. If $C$ consists of only one edge $u_1u_2$, then neither of these vertices can be the center of any other two-star-forest. Thus, the edge $u_1u_2$ was not covered by any two-star-forest $F_j$, which is a contradiction. Finally, if $C$ consists of two edges, $u_1u_2$ and  $u_1u_3$, say, then it is not difficult to see that at least one of the edges between $u_1,u_2,u_3$ in $K_n$ is not covered by any two-star-forest $F_j$. In each of the above cases, we obtained a contradiction. This completes the proof of Theorem~~\ref{Thm2}.   \end{proof}


In view of Theorem \ref{Thm2}, we state the following conjecture.

\begin{conj}\label{last}
For any $n \geq k\geq 2$, the number of $k$-star-forests needed to cover the complete graph $K_n$ is at least $\big \lceil \frac{(k+1)n}{2k} \big \rceil$.
\end{conj}

For $k=2$, the conjecture is true, by Theorem \ref{Thm2}. We construct an example inspired by the construction in \cite{AkK85}, showing that Conjecture~\ref{last}, if true, is best possible. 
For simplicity, we describe it only for the case where $n$ is divisible by $2$. 
Assuming $n=2t$, and labeling the vertices by $\{v_1,v_2,\cdots,v_n\}$, we create $t$ $2$-star-forests $F_1,F_2,\cdots F_t$ by picking vertices $v_i$ and $v_{i+t}$ as centers of $F_i$, $1\leq i \leq t$ and connecting $v_i$ to all vertices $v_j$, $i<j<i+t$, and connecting $v_{i+t}$ to all vertices $v_{j+t}$, $i<j<i+t$ (the indices are taken modulo n). The introduced $2$-star-forests cover all edges of $K_n$, except the set of edges $v_iv_{i+t}$ which can be simply decomposed into $\lceil \frac{n}{2k} \rceil$ $k$-star-forests.
Altogether, $K_n$ can be covered by $\frac{n}{2} + \lceil \frac{n}{2k} \rceil$ $k$-star-forests.



\begin{thebibliography}{8}

\bibitem{AOOPSSTV} O. Aichholzer, J. Obenaus, J. Orthaber, R. Paul, P. Schnider, R. Steiner, T. Taubner, and B. Vogtenhuber: Edge partitions of complete geometric graphs. \emph{38th International Symposium on Computational Geometry (SoCG 2022)} (2022)


\bibitem{AkK85} J. Akiyama and M. Kano:
Path factors of a graph. In: \emph{Graphs and Applications (Boulder, Colo., 1982)}, 
Wiley-Intersci. Publ., Wiley, New York, 1985, 1--21. 

\bibitem{ArDH05} G. Araujo, A. Dumitrescu, F. Hurtado, M. Noy, and J. Urrutia: On the chromatic number of some geometric type Kneser graphs, \emph{Comput. Geom.} {\bf 32} (2005), no. 1, 59--69.

\bibitem{AvH66}  S. Avital and H. Hanani: Graphs, continuation (in Hebrew), \emph{Gilyonot Le’matematika} {\bf 3} (1966), no. 2, 2--8.



\bibitem{BHRW06} P. Bose, F. Hurtado, E. Rivera-Campo, and David R. Wood: Partitions of complete geometric graphs into plane trees, \emph{Comput. Geom.}  {\bf 34} (2006), no. 2, 116--125.



\bibitem{DuW07}  V. Dujmovi\'c and D. R. Wood: Graph treewidth and geometric thickness parameters, \emph{Discrete Comput. Geom.} {\bf 37} (2007), no. 4, 641--670.


\bibitem{GrPo71} R. L. Graham and H. O. Pollak:
On the addressing problem for loop switching,
\emph{Bell System Tech. J.} {\bf 50} (1971), 2495--2519. 




\bibitem{Ku79} Y. S. Kupitz: \emph{Extremal Problems of Combinatorial Geometry, Lecture Notes Series} {\bf 53}, Aarhus University, Denmark, 1979.

\bibitem{Lo68} L. Lovász: On covering of graphs. In: \emph{Theory of Graphs (Proc. Colloq., Tihany, 1966),}  231–236 Academic Press, New York, 1968, 231--266.

\bibitem{Nas64} C. St. J. A. Nash-Williams:
Decomposition of finite graphs into forests,
\emph{J. London Math. Soc.} {\bf 39} (1964), 12.


\bibitem{Pa13} J. Pach:
The beginnings of geometric graph theory, in: \emph{Erd\H os centennial, Bolyai Soc. Math. Stud.} {\bf 25}, J\'anos Bolyai Math. Soc., Budapest, 2013, 465--484.




\bibitem{Vi64} V. G. Vizing: On an estimate of the chromatic class of a p-graph (in Russian),
\emph{Diskret. Analiz} (1964), no. 3, 25--30.

\end{thebibliography}
\end{document}